\documentclass{article}

\usepackage{ifxetex,ifluatex}
\if\ifxetex T\else\ifluatex T\else F\fi\fi T%
  \usepackage{fontspec}
\else
  \usepackage[T1]{fontenc}
  \usepackage[utf8]{inputenc}
  \usepackage{amsmath, latexsym,amsmath,amsfonts,amssymb, amsthm, graphicx, geometry, fullpage}
  \usepackage{color}
  \usepackage{amssymb}
\fi

\usepackage[hyperindex,breaklinks]{hyperref}
\hypersetup{
           breaklinks=true}
\usepackage{amsthm}

\usepackage[maxbibnames=99]{biblatex}
\usepackage{xurl}

\newcommand{\kommentar}[1]{}
\usepackage{color}
\definecolor{orange}{rgb}{0.7,0.3,0}
\definecolor{green}{rgb}{.4,.7,.4}
\usepackage{xcolor}
\definecolor{darkred}{HTML}{CC1F1F}
\definecolor{green}{rgb}{.4,.7,.4}
\definecolor{blue}{rgb}{.2,.6,.75}
\definecolor{pastelyellow}{rgb}{0.992157, 0.552941, 0.235294}
\definecolor{pastelorange}{rgb}{0.941176, 0.231373, 0.12549}
\definecolor{pastelred}{rgb}{0.741176, 0., 0.14902}
\definecolor{darkbrown}{rgb}{0.25098, 0., 0.0745098}
\usepackage{colortbl}

\theoremstyle{plain}
  \newtheorem{theorem}{Theorem}
  
   \newtheorem{thm}{Theorem}[subsection]
  
  \newtheorem{prop}[thm]{Proposition}
  \newtheorem{proposition}[thm]{Proposition}
  \newtheorem{cor}[thm]{Corollary}
  \newtheorem{lemma}[thm]{Lemma}
  
    \newtheorem*{question*}{Question}
\theoremstyle{definition}

\theoremstyle{remark}

\DeclareMathOperator{\Cl}{Cl}

\DeclareMathOperator{\Res}{Res}

\def\Z{{\mathbb Z}}

\def\bZ{{\mathbb Z}}

\def\Gal{{\rm Gal}}

\def\GL{{\rm GL}}

\def\Cl{{\rm Cl}}

\def\Ind{{\rm Ind}}

\def\Disc{{\rm Disc}}

\def\F{{\mathbb F}}

\def\bC{{\mathbb C}}
\def\bF{{\mathbb F}}

\def\Q{{\mathbb Q}}

\def\Z{{\mathbb Z}}

\def\F{{\mathbb F}}
\def\Q{{\mathbb Q}}

\def\bQ{{\mathbb Q}}

\def\cF{{\mathcal F}}

\def\bZ{{\mathbb Z}}

\def\bF{{\mathbb F}}
\def\bQ{{\mathbb Q}}
\def\cO{{\mathcal O}}

\def\cN{{\mathcal N}}

\def\fc{{\mathfrak c}}
\def\fp{{\mathfrak p}}

\def\fq{{\mathfrak q}}
\def\fm{{\mathfrak m}}
\def\fn{{\mathfrak n}}

\DeclareMathOperator{\Nm}{Nm}

\title{Power-saving error terms for the number of $D_4$-quartic extensions over a number field ordered by discriminant}
\author{Alina Bucur, Alexandra Florea, Allechar Serrano L\'opez, Ila Varma}

\begin{document}

\maketitle

\begin{abstract}
We study the asymptotic count of dihedral quartic extensions over a fixed number field with bounded norm of the relative discriminant. The main term of this count (including a summation formula for the constant) can be found in the literature (see \cite{cdod4} for the statement without proof and see \cite{Klunerswreath} for a proof), but a power-saving for the error term has not been explicitly determined except in the case that the base field is $\bQ$. In this article, we describe the argument for obtaining both the explicit main term and a power-saving error term for the number of $D_4$-quartic extensions over a general base number field ordered by the norms of their relative discriminants. We also give an extensive overview of the history and development of number field asymptotics.
\end{abstract}

\section{Introduction}

Obtaining asymptotics on the count of number fields ordered by discriminant has been a guiding force in the subfield of number theory that has been coined (originally by Mazur) as ``arithmetic statistics.'' 
 
The simplest case of this question arises from the study of quadratic fields: How many quadratic fields are there of bounded discriminant?  In this lowest-degree case, it is fortunate that quadratic fields are uniquely determined by their discriminant, and so answering the above question boils down to understanding the possible discriminants for quadratic fields. Algebraic methods lead to the conclusion that quadratic (or fundamental) discriminants are exactly the integers of the form $D$, where $D \equiv 1 \bmod 4$ and $D$ is square-free, or $4D$, where $D \equiv 2,3 \bmod 4$ and $D$ is square-free. Analytic methods can then be used to compute that the number of such integers with absolute value bounded by $X$ is
$$\frac{6}{\pi^2} X  + O(\sqrt{X}).$$
The above asymptotic can be proved by using contour integration and the fact that the Riemann zeta function $\zeta(s)$ does not vanish on $\Re(s)=1$.\footnote{
This result goes back many centuries, and was at least known to Gauss, who in \cite{disquisitiones} provides evidence that that class number of imaginary quadratic rings grows like $|D|^{1/2}$. In particular, Gauss claimed (without proof) that the sum of class numbers of quadratic fields of discriminant up to $T$ is $\sum_{|D| < T} H(D) \sim \frac{\pi}{18} T^{3/2},$ but Lagrange may have indexed these rings by their discriminant earlier. One can also think of this count as essentially reducing to counting squarefree integers up to $X$, which might have even been done by Euler.} 

In the 20th century, two main paths of generalizing this question on quadratic discriminants were studied: changing the degree of the number fields in question and changing the base field from $\bQ$ to a general number field (or function field). Altogether, we can pose the general question that is studied to understand number field asymptotics.

\begin{question*} Let $F$ be a base number field and fix an algebraic closure $\overline{F}$ of $F$. Fix a degree $n$ and a Galois group $G$ equipped with an embedding into $S_n$. What is the number of ($F$-isomorphism classes of) degree $n$  field extensions $L$ of $F$ such that the normal closure $\tilde{L}$ of $L$ over $F$ has Galois group isomorphic $G$ and the absolute norm of the relative discriminant $\Disc(L/F)$ is bounded by $X$?\end{question*}

Here, we say two extensions $L$ and $L'$ of a number field $F$ are {\em $F$-isomorphic} if there exists a field isomorphism between $L$ and $L'$ that restricts to the identity on $F$. 

\medskip

Let $N_{F,n}(G,X)$ denote the number of $F$-isomorphism classes of degree $n$ extensions $L$ over $F$ such that $\Gal(\tilde{L}/F) \cong G$ and $\Nm_{F/\bQ}(\Disc(L/F)) \leq X$. We also require that the Galois group $\Gal(\tilde{L}/F)$ of the Galois closure $\tilde{L}$ over $F$, viewed as permuting the set of embeddings of $L$ into $\tilde{L}$, is isomorphic to the fixed choice of image of $G$ in $S_n$. 
Our focus is then on determining the constants of the main term of the asymptotic formula for $N_{F,n}(G,X)$; that is, finding $a_{n}(G), b_{F,n}(G)$ and $c_{F,n}(G)$ such that
\begin{equation}\label{nfasymptotics}
\lim_{X\rightarrow \infty} \frac{N_{F,n}(G,X)}{X^{a_{n}(G)} \log (X)^{b_{F,n}(G)}} = c_{F,n}(G) .
\end{equation}
The study of the positive rational constants $a_{n}(G) = a_{F,n}(G)$ and $b_{F,n}(G)$ is the focus of Malle's Conjecture \cite{malle1, malle2}. We first discuss the development of results and techniques for number field asymptotics predating the statement of Malle's Conjecture. After introducing our main result, we discuss the evolution of the study of number field asymptotics after Malle's statement of their conjecture.

\subsection{Number field asymptotics predating Malle's Conjecture}
Much of the work predating the statement of Malle's conjecture focused on the case when $F=\bQ$. Past the $n = 2$ case described above, Cohn \cite{cohn} determined in 1954 using class field theory and generating functions that equation \eqref{nfasymptotics} holds when $F = \bQ$, $n = 3$, and $G = C_3$ with 
$$a_{\bQ,3}(C_3) = 1/2, \quad  b_{\bQ,3}(C_3) = 0, \quad  c_{\bQ,3}(C_3) = \frac{11 \sqrt{3}}{36 \pi} \cdot \prod_{p \equiv 1 \bmod 6} \frac{(p+2)(p-1)}{p(p+1)}.$$ In 1971, Davenport--Heilbronn \cite{davenportheilbronn} complemented this work by determining that the asymptotic in \eqref{nfasymptotics} holds for $F = \bQ$, $n = 3$ and $G = S_3$ with  $$a_{\bQ,3}(S_3) = 1, \quad b_{\bQ,3}({S_3}) = 0, \quad c_{\bQ,3}(S_3) = \frac{1}{3 \zeta(3)}.$$ 
Their results utilized an algebraic parametrization of cubic rings in terms of $\GL_2(\bZ)$-orbits of integer binary cubic forms; geometry-of-numbers techniques to relate the count of certain orbits to a volume of a body in the fundamental domain for $\GL_2(\bZ)$ on real-coefficient binary cubic forms; and finally sieving methods to restrict to maximal rings and thus obtain an asymptotic count for $S_3$-cubic fields.

In 1979, Baily \cite{baily} proved the asymptotic \eqref{nfasymptotics} in the abelian degree 4 cases ($F = \bQ$, $n = 4$, and $G = C_4$ or $V_4$) with
$$a_{\bQ,4}(C_4) = 1/2,\,\, b_{\bQ,4}(C_4) = 0,\,\, c_{\bQ,4}(C_4) = \frac{3}{\pi^2}\left( -1+ \frac{25}{24}\prod_{p\equiv 1 \bmod 4} \left(1 + \frac{2}{p^{3/2} + p^{1/2}}\right)\right),$$ and
$$a_{\bQ,4}(V_4) = 1/2,\,\, b_{\bQ,4}(V_4) = 2, \,\,c_{\bQ,4}(V_4) =  \frac{1}{48\zeta(2)^3} \prod_p \left(1 + \frac{3p+1}{p^3 + 3p^2}\right)^{-1}.$$
Baily used class field theory and generating functions in a similar manner as the cyclic cubic case for the above results.

In addition, Baily proved that $c X < N_{\bQ,4}(D_4,X)< c' X$ for two constants $c$ and $c',$ where $D_4$ denotes the symmetries of the square. The upper bound is obtained using bounds on the number of quadratic ray class characters of quadratic fields which correspond under class field theory to $D_4$-quartic fields and summing these bounds over all quadratic fields with bounded discriminant. The lower bound was obtained by studying quadratic ray class characters of certain imaginary quadratic fields with even class number. 

In the 1970s and early 1980s, a variety of authors worked on asymptotics for various families of $p$-groups, such as the case $G = C_{p^k}$ for $p \geq 3$, $k \in \bZ_{\geq 2}$, $F = \bQ$ and $n = p^k$, as well as $G = (C_p)^k$ and direct products of $p$-groups with $p'$-groups, where $p,p' \geq 3$ are distinct primes (see \cite{urazbaev,zhang,varmon,oriat}. For more detailed lists, see Pg. 7 of \cite{maki}).

In the late 1980s, a lot of different avenues saw progress, mostly surrounding work of M\"aki and Wright. In 1985, M\"aki \cite{maki} determined that for any abelian group $G$ of order $n = |G|$,
  $$a_{\bQ,|G|}(G) = \frac{1}{n}\left(1 - \frac{1}{p}\right), \quad b_{\bQ,|G|}(G) = \begin{cases} \frac{n_p}{p-1} - 1 & \mbox{if }p \geq 3 \\ n_p - 1 & \mbox{if } p = 2\end{cases}$$
 where $p$ is the smallest prime dividing $n$ and $n_p$ denotes the number of elements of $G$ of order $p$. 
In 1988, Datskovsky--Wright \cite{datskovskywright} determined for any number field $F$,
$$a_{F,3}(S_3) = 1,\,\, b_{F,3}(S_3) = 0, \,\, c_{F,3}(S_3) = \left(\frac{2}{3}\right)^{r_1(F) - 1}\left(\frac{1}{6}\right)^{r_2(F)} \frac{\zeta_F(1)}{\zeta_F(3)},$$ where $\zeta_F(1)$ denotes the residue at $s = 1$ of $\zeta_F(s)$, the Dedekind zeta function associated to $F$, $r_1(F)$ denotes the number of real embeddings of $F$ into $\overline{F}$, and $r_2(F)$ denotes the number of pairs of complex embeddings of $F$ into $\overline{F}$. They also obtained analogous results for any global field of characteristic not equal to 2 or 3; see \cite[Theorem I.1]{datskovskywright}. As a byproduct, they  determined 
$$a_{F,2}(S_2) = 1,\,\, b_{F,2}(S_2) = 0,\,\, c_{F,2}(S_2) = 2^{-r_2(F) - 1} \frac{ \zeta_F(1) }{\zeta_F(2)}.$$ 
The authors accomplished this work building on the work of Davenport and Heilbronn \cite{davenportheilbronn} and through the use of the adelic Shintani zeta functions \cite{Shintani1,Shintani2}.

Building on work of M\"aki, Wright \cite{wright} proved that for any finite abelian group $G$ of cardinality $n$ and any number field $F$, there exists a constant $c_{F,n}(G) > 0$ such that the asymptotic \eqref{nfasymptotics} holds with the exponents
$$a_{n}(G) = \frac{1}{n}\left(1-\frac{1}{p}\right),\,\, b_{F,n}(G) = \frac{n_p}{[F(\zeta_p):F]} - 1 ,$$ where $p$ is the smallest prime dividing $n$, and $n_p$ denotes the number of elements of $G$ of order $p$. Wright developed general techniques in the vein of \cite{cohn} and \cite{baily}, i.e. incorporating class field theory with the use of generating functions that happen to be related to Dirichlet series. Standard Tauberian theorems (most notably, Perron's formula) then allow for the extraction of number field asymptotics.

In the the early 2000s, Cohen--Diaz y Diaz--Olivier \cite{cdod4} determined that, for $F=\bQ$, $n = 4$, and $G=D_4$, 
\begin{equation}\label{CDOD4}
a_{\bQ,4}(D_4) = 1 \quad b_{\bQ,4}(D_4) = 0 \quad c_{\bQ,4}(D_4) = \sum_{[K:\bQ] = 2} \frac{2^{-r_2(K)-1}}{\Disc(K)^2}\frac{\zeta_K(1)}{\zeta_K(2)}.
\end{equation}
They utilized Kummer theory in conjunction with the strategy of Baily \cite{baily} to relate the count of $D_4$-quartic extensions to the count of quadratic extensions of the quadratic extensions. The authors of \cite{cdod4} define a Dirichlet series associated to counting quadratic extensions of quadratic extensions and relate it to the generating function for $D_4$-quartic fields.  
They also remark that it is immediate to extend this result to the case of relative $D_4$-extensions over a general base field.\

\subsection{Main Results}

In this paper, we give a proof of the asymptotic \eqref{nfasymptotics} in the case of quartic dihedral extensions of a number field $F$, with $a_{F,4}(D_4)=1$, $b_{F,4}(D_4) = 0$, and determine the constant $c_{F,4}(D_4)$. Further, we obtain an explicit formula for a power-saving error term in terms of the degree of $F$. 
 More precisely, we prove the following result.

\begin{theorem}\label{main} Let $F$ be a number field with $[F:\bQ] = d$. We have that the asymptotic number of ($F$-isomorphism classes of) quartic $D_4$-extensions of $F$ whose norm of the relative $F$-discriminant is bounded by $X$ is, as $X$ grows, 
  \begin{equation}\label{eqmain}N_{F,4}(D_4,X) = \left(\sum_{[K:F]=2}\frac{1}{2^{r_2(K)+1}} \cdot \frac{1}{\Nm_{F/\bQ}(\Disc(K/F))^2} \cdot \frac{\zeta_{K}(1)}{\zeta_{K}(2)}\right)\cdot X + O\left(X^{\frac{3}{4}+\epsilon}\right).\end{equation}
\label{rem1}
If we assume the Lindel\"{o}f hypothesis, then the error term above can be improved to $O\left(X^{\frac{1}{4}+\epsilon}\right).$
\end{theorem}
Recall that for $\Re(s)>1$ the Dedekind zeta function of any number field $E$ is given by \[\zeta_E(s) = \prod_{\fp \subseteq \cO_E} \left({1 - \Nm_{E/\bQ}(\fp)^{-s}}\right)^{-1},\] where $\fp$ runs over prime ideals of $\cO_E$. It has meromorphic continuation to the full complex plane, and is analytic everywhere except for a simple pole at $s = 1$. By abuse of notation, we denote by $\zeta_E(1)$ the residue at $s = 1$.

The main term in Theorem \ref{main} is proven by Kl\"uners \cite{Klunerswreath} without a power-saving error term. We show that the proof argument utilized by Cohen--Diaz y Diaz--Olivier \cite{cdod4} to prove that $a_{\bQ,4}(D_4) = 1$ and $b_{\bQ,4}(D_4) = 0$ holds true for general $F$, and that the formula \eqref{eqmain} is the same as the one stated by them in Remark (4) after \cite[Corollary 6.1]{cdod4} with $\alpha =  \frac{3}{4}$. We point out that our power-saving error term does not depend on the degree of $F$, as the aforementioned authors suggest!

\subsection{Malle's conjecture, counterexamples, refinements, and proven cases}

In 2004, Malle \cite{malle1, malle2} made a prediction for the constants $a_{F,n}(G) = a_{n}(G)$ and $b_{F,n}(G)$ for all transitive groups $G \hookrightarrow S_n$. The prediction was that $$a_{n}(G) = \frac{1}{\min_{g \in G \backslash 1} \Ind (g)},$$ 
where the index of an element $g$ is defined as $\Ind(g) := n - \#$ orbits of $\{1,2,...n\}$ when acted upon by $g$ via the embedding into $S_n$,
and 
$$b_{F,n}(G) = b^{\rm Malle}_{F,n}(G) := \#\{\Gal(\overline{F}/F)\text{- orbits of } G\text{-conjugacy classes of minimal index } \Ind(C)\} -1,$$  where the notation is defined as follows: Note that if $g$ and $g'$ are in the same conjugacy class $C$ of $G$, then $\Ind(g) = \Ind(g')$, and furthermore, the absolute Galois group $\Gal(\overline{F}/F)$ acts on the set of conjugacy classes $\{C\}$ of $G$ via the action of the cyclotomic character, i.e., $\mu: \Gal(\overline{\bQ}/\bQ) \rightarrow \hat{\bZ}^\times$ so that $\sigma \in \Gal(\overline{\bQ}/\bQ)$ acts on roots of unity as $\sigma(\zeta_n) = \zeta_n^{\mu(\sigma)}$. For $g \in G$, we can define $\sigma(g) = g^{\mu(\sigma)}$, or in fact on conjugacy classes $C$ of $G$: $\sigma(C) = C^{\mu(\sigma)}$. We call the orbits under this action on the $G$-conjugacy classes $\{C\}$ by $\Gal(\overline{F}/F)$-orbits. 

Note that in the trivial case when $n=1$ and $G$ is the trivial group we have $a_{1,F}(1) = b_{1,F}(1) = 0$ and $c_{1,F}(1) = 1$ for any field $F.$ 

It is also worth pointing out that when the image of $G$ in $S_n$ contains a simple transposition $(ab)$, then $\Ind((ab)) = 1$ since $(ab)$ has $n-1$ orbits in $S_n$, hence $a_n(G)$ is predicted to be 1. The number of $G$-conjugacy classes of minimal index $1$ will be $1$, since all simple transpositions are conjugate to one another and no other permutation type can have $n-1$ orbits, hence $b_{F,n}^{\rm Malle}(G) = 0$ in this case. 

In 2005, two different directions of progress arose from the development of averaging techniques in geometry-of-numbers counting due to Bhargava and a noteworthy counterexample and subsequent work to Malle's Conjecture due to Kl\"uners. 

Bhargava \cite{densityofquartics} determined for $S_4$-quartic fields over $\bQ$ that $a_{\bQ,4}(S_4) = 1$ and $b_{\bQ,4}(S_4) = 0$ and obtained an Euler product formula for 
$$c_{\bQ,4}(S_4) = \frac{5}{24}\prod_p \left(1 + \frac{1}{p^2} - \frac{1}{p^3} - \frac{1}{p^4}\right).$$ 
Bhargava's methods built on the work of Davenport-Heilbronn \cite{davenportheilbronn}, Delone--Faddeev \cite{DF}, and Sato-Kimura \cite{SK} using the same proof strategy of a parametrization (this time, in terms of pairs of ternary quadratic forms) allowing for geometry-of-numbers methods followed by sieving techniques. However, Bhargava produced a general machine of averaging techniques for cutting off cusps when computing a volume in the fundamental domain and relating it to the number of integral orbits or lattice points that enumerate quartic rings. This allowed for generalizations to the $n = 5$ case \cite{densityofquintics}, in which he determined in 2010 that $a_{\bQ,5}(S_5) = 1$ and $b_{\bQ,5}(S_5) = 0$ and obtained an Euler product formula for 
$$c_{\bQ,5}(S_5) =  \frac{13}{120} \prod_p \left(1 + \frac{1}{p^2} - \frac{1}{p^4} - \frac{1}{p^5}\right).$$ In addition, Bhargava and Wood \cite{bhargavawood} and independently, Belabas and Fouvry \cite{belabasfouvry} verified for $S_3$-sextic fields that $a_{\bQ,6}(S_3) = 1/3$ and $b_{\bQ,6}(S_3)=0$ and furthermore determined that $$c_{\bQ,6}(S_3) = \frac{1}{3}\left(1+ \frac{1}{3} + \frac{1}{3^{5/3}} + \frac{2}{3^{7/3}}\right)\left(1 - \frac{1}{3}\right)\prod_{p \neq 3} \left(1 + \frac{1}{p} +\frac{1}{p^{4/3}}\right)\left(1-\frac{1}{p}\right) .$$

While studying these questions, Bhargava \cite{BhargavaIMRN} formulated a conjecture for the constant in the case when $G = S_n$. Namely, he conjectured that $$c_{F,n}(S_n) = \frac{1}{2}\cdot \zeta_F(1) \left(\frac{\#S_n[2]}{n!}\right)^{r_1(F)} \left(\frac{1}{n!}\right)^{r_2(F)} \prod_\fp \left(\sum_{k=0}^n \frac{q(k,n-k) - q(k-1,n-k+1)}{\Nm_{F/\bQ}(\fp)}\right) $$ for all $n$, where $q(k, n)$ denotes the number of partitions of $k$ into at most $n$ parts (see conjecture A in the Appendix of \cite{BhargavaIMRN}). 
A conjecture for $c_{F,n}(G)$ for general $G$ has not been formulated as the direct analogy to Bhargava's heuristic is contradicted by the result of \cite{cdod4} in the case of $n = 4$ and $G = D_4$ (see \cite{cdod4} and \cite[Section \S3.3]{ASVW}). Bhargava's quartic and quintic results were generalized to all number fields $F$ by  Bhargava--Shankar--Wang \cite{bsw}, including formulas for the constants $c_{F,4}(S_4)$ and $c_{F,5}(S_5)$ that verified Bhargava's conjecture along with Malle's conjecture. 

Also in 2005, Kl\"uners \cite{Klunerscounterexample} provided a counterexample to Malle's Conjecture by proving a lower bound for $N_{\bQ,6}(C_3,\wr C_2,X)$ that implies that $b_{\bQ,6}(C_3 \wr C_2) > 0$, contradicting the formula given by Malle, which would have predicted that $b_{\bQ,6}({C_3 \wr C_2}) = 0$.  The contradiction seems to arise from the fact that while there aren't third roots of unity in the base field, third roots of unity can arise in an intermediate extension, and $C_3 \wr C_2$-extensions of $\bQ$ that contain $\bQ(\zeta_3)$ and have discriminant bounded by $X$ are themselves at least $c \cdot X^{1/2}\log(x)$ for some positive constant $c$. Nevertheless, Kl\"uners \cite{Klunerswreath} did prove that Malle's conjecture holds for $G = C_2 \wr H$ where $H$ is a transitive subgroup of $S_{n/2}$ such that over $F$, there is at least one extension of $F$ with Galois group $H$ and $N_{F,n/2}(H,X) = O_{F,H,\epsilon}(X^{1+\epsilon})$ for every $\epsilon > 0$. Building on \cite{cdod4}, Kl\"uners \cite{Klunerswreath} computed  $$c_{F,n}(C_2 \wr H) = \sum_{{[K:F]=n/2 \atop \Gal(\tilde{K}/F) = H}} \frac{1}{2^{r_2(K)}\Nm_{F/\bQ}(\Disc(K)^2)} \cdot\frac{\zeta_K(1)}{\zeta_K(2)}.$$ 
In particular, when $H = C_2$, Kl\"uners \cite{Klunerswreath} proved the generalization of \eqref{CDOD4} in the $D_4$-quartic case from over $\bQ$ to over a general base number field $F$ without a power-saving error term. Furthermore, Kl\"uners \cite{Klunersthesis} verified Malle's Conjecture for generalized quaternion groups of order $4m$, determining that $a_{F,4m}(Q_{4m}) = 1/2m$ and $b_{F,4m}(Q_{4m}) = 0$.

In 2015, T\"urkelli \cite{turkelli} formulated function field heuristics for $N_{F,n}(G,X)$ (i.e., when $F$ is taken to be a function field $\bF_q(t)$ instead of a number field), and came up with a way to salvage Malle's conjecture that is consistent with Kl\"uners's counterexample. Namely, consider first for a group $G$ the set $\cN(G)$ of normal subgroups $N$ in $G$ such that $G/N$ is abelian and $a_{n}(N) = a_{n}(G)$ (and where $a_{n}(\{1\}) = 0$). Define $$b_{F,n}(G,N) = \max_{\varphi \in {\rm Sur}(\Gal(F(\mu_\infty)/F),G/N))} b_{F,n}(G,N,\varphi),$$ where $b_{F,n}(G,N,\varphi)$ is the number of twisted $F$-orbits of $N$-conjugacy classes of minimal index, where the action of $\Gal(\overline{F}/F)$ is defined now with a twist: for a conjugacy class $C$ of $N$, $\sigma(C) = \varphi(\Res(\sigma))^{-1} \cdot C^{\mu(\sigma)}$, where $\Res: \Gal(\overline{F}/F) \rightarrow \Gal(F(\mu_\infty)/F)$ and $\varphi(\Res(\sigma))^{-1}$ acts by conjugation on the $N$-conjugacy class $C^{\mu(\sigma)}$.\footnote{It is worth pointing out that in Malle's original formulation, he only considered the case of $N = G$, and so $b^{\rm Malle}_{F,n}(G) = b_{F,n}(G,G)$. This only coincides with $b_{F,n}(G)$ when $b_{F,n}(G,G)$ is the maximum over all $b_{F,n}(G,N)$ for $N \in \cN(G)$.} Then, we can define $$b_{F,n}(G) = -1 +\max_{N \in \cN(G)} b_{F,n}(G,N)$$ as the salvage to $b_{F,n}^{\rm Malle}(G)$.  

T\"urkelli's heuristics \cite{turkelli} seemed to contradict earlier predictions made by Ellenberg--Venkatesh \cite{ellenbergvenkatesh1} for counting extensions of the function field $\F_q(t)$ which agree with Malle's original conjecture. However, Ellenberg--Venkatesh focused on extensions that do not contain any nontrivial constant intermediate extensions, while T\"urkelli \cite{turkelli} allowed for those to be counted in their heuristic, which has a stronger parallel to Malle's conjecture where extensions containing cyclotomic subfields are included in the asymptotics. (On the geometric side, T\"urkelli \cite{turkelli} includes connected, but not geometrically connected, covers of the projective line over $\F_q$ while Ellenberg-Venkatesh \cite{ellenbergvenkatesh1} require covers to be geometrically connected as well.)

More recently, Wang \cite{wangthesis} has made substantial progress on techniques for proving Malle's conjecture in compositums of number fields. The first application of these methods was to prove Malle's Conjecture for certain direct product groups of the form $G = S_n \times A$, including the case when $n = 3$ and $A$ is an odd abelian group; the case when $n = 4$ and $A$ is an abelian group of order coprime to 6; and the case when $n = 5$ and $A$ is an abelian group of order relatively prime to 30 (see \cite{wangcompositio}). Later generalizations came from joint work of Masri--Thorne--Tsai--Wang \cite{mttw} to proving Malle's conjecture for $G = S_n \times A$ where $A$ is any abelian group and $n = 3,4,5$.

Most recently, Fouvry--Koymans \cite{fouvrykoymans} proved Malle's conjecture for nonic Heisenberg extensions by estimating a certain character sum. They determined that there exists a constant $c_{\bQ,9}(\mathrm{Heis}_3)>0$, which they determine explicitly (see Pgs. 17-18 of \cite{fouvrykoymans}), such that \eqref{nfasymptotics} holds with
$$a_{\bQ,9}(\mathrm{Heis}_3) = 1/4,\,\, b_{\bQ,9}(\mathrm{Heis}_3) = 0.$$
Additionally, Koymans--Pagano \cite{koymanspagano} proved that for any number field $F$ and any finite nilpotent group $G$ of cardinality $n$ such that all elements of order $p$ are central, where $p$ is the smallest prime dividing $n$, there exists a constant $c_{F,n}(G) > 0$, and determined 
$$a_{n}(G) = \frac{1}{n}\left(1-\frac{1}{p}\right),\,\, b_{F,n}(G) = \frac{n_p}{[F(\zeta_p):F]} - 1 ,$$ where $n_p$ denotes the number of elements of $G$ of order $p$.

At the time of writing this article, the authors have heard of forthcoming work of Alberts--Lemke Oliver--Wang--Wood \cite{albertsetal} in which the authors develop inductive counting methods to prove Malle's conjecture for groups $G$ for which all elements in $G$ of sufficiently small index are contained in a normal subgroup which is either abelian or generated by conjugates of a single subgroup of $G$ isomorphic to $S_3$. Additionally, work-in-progress at the time of writing this article includes Shankar--Varma \cite{shankarvarma} for verifying Malle's conjecture in the Galois octic case by determining $a_{\bQ,8}(D_4) = 1/4$ and $b_{\bQ,8}(D_4) = 2$ and computing an Euler product for $c_{\bQ,8}(D_4)$. 

We do not claim that this is a full history of the study of number field asymptotics; we tried to focus on the results that specifically proved (the strong form of) Malle's conjecture and study of the constants $c_{F,n}(G)$ in those cases. Much work has been developed in proving the weak form of Malle's conjecture \cite{malle1}, i.e., that there exist positive real constants $c^{(1)}_{F,n}(G)$ and $c^{(2)}_{F,n}(G,\epsilon)$ such that
  $$c^{(1)}_{F,n}(G) X^{a(G)} \leq N_{F,n}(G,X) < c^{(2)}_{F,n}(G,\epsilon) X^{a(G)+\epsilon} $$
for large enough $X$ (see \cite{klunersmalle}, \cite{alberts1}, \cite{alberts2}). Many analogous results have been obtained over function fields, the case of quartic $D_4$-extensions by Keliher \cite{keliher}, the case of cubic $S_3$-extensions studied by Wood \cite{woodtrigonal}, work of Ellenberg--Venkatesh--Westerland \cite{ellenbergetal} on degree $p$ $D_p$-extensions for odd primes $p$, the case of $C_2$-extensions (with added local conditions) studied by Kurlberg--Rudnick \cite{kr}, the case of degree $p$ $C_p$-extensions (with local conditions) for odd primes $p$ studied by Bucur--David--Feigon--Lal\'in \cite{bdfl} and Bucur--David--Feigon--Kaplan--Lal\'in--Ozman--Wood \cite{bucuretal}.

In addition, much work has been done in proving and improving upper and lower bounds for $N_{F,n}(G,X)$ (see \cite{schmidt}, \cite{shih}, \cite{Klu06}, \cite{evbounds}, \cite{dummit}, \cite{pierceetal}, \cite{klunerswang}, \cite{mallematzat}, \cite{matzatmarschke}, \cite{bswschmidt}, \cite{rlothorne}, \cite{improvedsmalldeg})
as well as other important work on building methods and techniques for answering these questions.

\subsection{Summary of our methods}

Fix a number field $F$ of degree $d$, and let $\overline{F}$ denote a fixed algebraic closure of $F$. In the sequel, all extensions of $F$ will be assumed to be contained in $\overline{F}$. 

Recall that $N_{F,n}(G,X)$ denotes the number of $F$-isomorphism classes of degree $n$ extensions $L$ over $F$ such that the normal closure $\tilde{L}$ of $L$ over $F$ has Galois group $\Gal(\tilde{L}/F) \cong G$ and $\Nm_{F/\bQ}(\Disc(L/F)) \leq X$, and we say two extensions $L$ and $L'$ of a number field $F$ are {\em $F$-isomorphic} if there exists a field isomorphism between $L$ and $L'$ that restricts to the identity on $F$.

We are interested in counting the number of ($F$-isomorphism classes of) quartic extensions $L$ whose normal closure over $F$ has Galois group $D_4$, the symmetries of a square, when such extensions are ordered by the norms of their relative discriminants over $F$.  We will call such an $L$ a {\em $D_4$-quartic extension} over $F$, and let $\Disc(L/F)$ denote the relative discriminant. Denote by $\cF_{F,4}(D_4)$ the set of $F$-isomorphism classes of $D_4$-quartic extensions over $F$. 

To obtain the order of magnitude for $D_4$-quartic extensions over a general base number field $F$ with discriminant norm bounded by $X$, we must understand the poles of the Dirichlet series $\Phi_{F,4}(D_4,s)$ associated to the count $N_{F,4}(D_4,X)$:
$$\Phi_{F,4}(D_4,s) = \sum_{L \in \mathcal{F}_{F,4}(D_4)} \frac{1}{\Nm_{F/\Q}(\Disc(L/F))^{s}},$$
where $\cF_{F,4}(D_4)$ denotes the set of $F$-isomorphism classes of $D_4$-quartic extensions of $F$. Perron's formula implies that
\begin{equation}
N_{F,4}(D_4,X) = \frac{1}{2 \pi i} \int_{\Re(s) =2} \Phi_{F,4}(D_4,s) \frac{X^s}{s}\, ds.
\end{equation}
Determining the residue of the poles of the above integral after shifting the contours of integration gives the main term contribution for $N_{F,4}(D_4,X)$, and we additionally obtain (a power-saving) error term from bounding the contribution of the integrals on the other sides of the contour of integration. 

To study the residue, we follow the same strategy as Cohen--Diaz y Diaz--Olivier utilized in \cite[Corollary 1.4]{cdod4} to prove the case $F = \bQ$. Specifically, we will relate the count $N_{F,4}(D_4,X)$ to the count of quadratic extensions of quadratic extensions of $F$, and therefore relate the Dirichlet series $\Phi_{F,4}(D_4,s)$ to a Dirichlet series
  $$\frac{1}{2}\Phi(s) =\frac{1}{2}\sum_{[K:F]=2} \frac{1}{\Nm(\Disc(K/F))^{2s} }\left(1+\sum_{[L:K] = 2}\frac{1}{\Nm(\Disc(L/K))^{2s}} \right),
$$
which has the same residue at the simple pole at $s = 1$ as $\Phi_{F,4}(D_4,s)$. We can then use Cohen--Diaz y Diaz--Olivier's results on counting on quadratic extensions of a number field to give a summation formula for the constant that verifies \cite[Remark 4]{cdod4}.

\section{Notation and Background}
In this section, we introduce the notation and background information we will be using throughout the article. 

For any number field $E$, let $\cO_E$ denote its ring of integers, let $r_1(E)$ denote the number of real embeddings of $E$, and let $r_2(E)$ denote the number of pairs of complex conjugate embeddings.

\subsection{$L$-functions associated to a number field}
For $s \in \bC$ such that $\Re(s) > 1$, the Dedekind zeta function of a number field $E$ is given by the Euler product running over nonzero primes ideals of $\cO_E$
$$\zeta_{E}(s) = \prod_{\fp \subseteq \cO_E} \frac{1}{1 - \Nm_{E/\bQ}(\fp)^{-s}}.$$ Hecke proved that $\zeta_E(s)$ has meromorphic continuation to the entire complex plane with only a simple pole at $s = 1$. 

For a character $\chi$ of a ray class group $\mathcal{C}\ell_{\mathfrak{m}}(F)$ of conductor $\fm$ of $E$, let $L_E(s,\chi)$ denote the $L$-function associated to the character $\chi$, namely for $\Re(s)>1$,
$$L_E(s,\chi) = \sum_{\substack{\mathfrak{a} \subset \mathcal{O}_E \\ (\mathfrak{a},\mathfrak{m})=1}} \frac{\chi(\mathfrak{a})}{\Nm_{E/\bQ}(\mathfrak{a})^s}.$$ 
\begin{lemma} \label{2.1.1} 
If $\chi=\chi_0$ is the trivial character, then $L_E(s,\chi_0)$ has a simple pole at $s= 1$. Otherwise, $L_{E}(s,\chi)$ is analytic everywhere. 
\end{lemma}
\begin{proof}
We see that 
\begin{equation}
\label{l-zeta}
L_E(s,\chi_0) = \zeta_E(s)\prod_{\fp \mid \fm} \left(1 - \Nm_{E/\bQ}(\fp)^{-s}\right),
\end{equation}
and the conclusion follows from the properties of $\zeta_E(s)$ as the product over $\fp \mid \fm$ is finite. For nontrivial $\chi$, the result follows from \cite[Chapter VII, Theorem 8.5]{neukirch}.
\end{proof}

If there is a pole at $s = s_0$, and we write $L_E(s_0,\chi)$ or $\zeta_E(s_0)$, we are abusing notation and referring to the residue at the pole. 

The Lindel{\"o}f hypothesis would give that
\begin{equation}
\zeta_E \left(\frac{1}{2} + it\right) \ll t^{\epsilon},
\label{lindelof=1}
\end{equation} and a similar bound holds for $s=\sigma + it$ and $\sigma > \frac{1}{2}$.

\subsection{Counting $G$-extensions of a number field and associated Dirichlet series}
Fix a finite transitive subgroup $G$  of $S_n$ for some $n \geq 2$. Let $\mathcal{F}_{E,m}(G)$ denote the set of $E$-isomorphism classes of extensions $E'/E$ of degree $m$ (in a fixed algebraic closure of $E$) such that the Galois group of the Galois closure of $E'/E$ is isomorphic to $G$.

Denote by $N_{E,m}(G,X)$ the number of elements of $\mathcal{F}_{E,m}(G)$ whose discriminant is bounded by $X$; namely
$$N_{E,m}(G,X) =\# \{ E' \in \mathcal{F}_{E,m}(G) ; \, \Nm_{E/\bQ}(\Disc(E'/E)) \leq X\}. $$

In order to obtain asymptotic formulas for $N_{E,m}(G,X)$, we need to understand the analytic properties of the function $\Phi_{E,m}(G,s)$ defined by 
\[\Phi_{E,m}(G,s) = \sum_{E' \in \mathcal{F}_{E,m}(G)} \frac{1}{\Nm_{E/\Q}(\Disc(E'/E))^{s}}.\]

Using Perron's formula, we will obtain the main term in the asymptotic formula for $N_{E,m}(G,X)$ from the right-most pole of $\Phi_{E,m}(G,s)$, while the error term will be dictated by how much information we have about the meromorphic continuation of $\Phi_{E,m}(G,s)$.

\section{Counting quadratic extensions of quadratic extensions of a number field}

For the rest of the paper, we fix a number field $F$ and set $d= [F:\Q].$ We are interested in computing an asymptotic formula for $N_{F,4}(D_4,X)$, and our strategy is to study $\Phi_{F,4}(D_4,s)$ by relating it to the Dirichlet series associated to counting quadratic extensions of quadratic extensions of $F$ of bounded discriminant norm. The equivalent of $\Phi_{F,m}(G,s)$ for quadratic extensions of quadratic extensions is the function 
\begin{equation}\tilde{\Phi}(s) = \sum_{[K:F]=2} \sum_{[L:K]=2} \frac{1}{\Nm_{F/\Q}(\Disc(L/F))^{s}}\label{phitilde}.
\end{equation}

We will first study the relationship of $\Phi_{F,4}(D_4,s)$ and $\tilde{\Phi}(s)$, showing that they only differ by a holomorphic function. We will then study the poles and residues of $\tilde{\Phi}(s)$ by instead considering a slightly altered Dirichlet series $\Phi(s)$ that again only differs from $\tilde{\Phi}(s)$ (and therefore also from $\Phi_{F,4}(D_4,s)$) by a holomorphic function. 

\subsection{Quadratic extensions of quadratic extensions}

For a quadratic extension $K$ of $F$, let $\cF_2(K)$ denote the $K$-isomorphism classes of quadratic extensions $L$ of $K$. There are three possibilities for the extension $L/F$ in this case:
\begin{itemize}
\item $L/F$ is a degree $4$ non-Galois extension whose normal closure has Galois group $D_4$ over $F$. These are the extensions we want to count.
\item $L/F$ is a Galois extensions with Galois group $C_4$, the cyclic group of order $4.$
\item $L/F$ is a Galois extension with Galois group $V_4 \simeq \Z/2 \times \Z/2$, the Klein four-group.
\end{itemize}

By \cite[Corollary 2.3(3)]{cdod4},
\begin{equation}\label{counting}
\sum_{K \in \cF_2(F)} \sum_{\substack{L \in \cF_2(K) \\ \Nm_{F/\bQ}(\Disc(L/F)) \leq X}} 1  \qquad = \quad 2N_{F,4}(D_4,X) + N_{F,4}(C_4,X) + 3N_{F,4}(V_4,X)
\end{equation}

To see this note that for a fixed quadratic extension $K$ of $F$, conjugate $D_4$-extensions of $F$ will be counted separately by the sum on the left, and every $D_4$-extension contains a single quadratic subextension (shared with its unique $F$-conjugate). The sum on the left counts $C_4$-extensions up to $F$-isomorphism just once since these extensions are Galois and contain a single quadratic subextension. Each $V_4$-extension $L/F$ contains three quadratic subfields, so the sum on the left is counting each $V_4$-extension up to $F$-isomorphism three times.

Hence, we can express the function $\Phi_{F,4}(D_4,s)$ we are interested in as 
\begin{equation}
\Phi_{F,4}(D_4,s) = \frac{1}{2} \left(\tilde{\Phi}(s)-\Phi_{F,4}(C_4,s)-3 \Phi_{F,4}(V_4,s) \right).
\label{phij}
\end{equation}

On the other hand, note that by the relative discriminant formula, if $K \in \cF_2(F)$ and $L \in \cF_2(K)$, then $\Disc(L/F) = \Disc(K/F)^2 \Nm_{K/F}(\Disc(L/K))$. Thus, we can write
\begin{align}\tilde{\Phi}(s) &= \sum_{[K:F]=2} \sum_{[L:K]=2} \frac{1}{\Nm_{F/\Q}(\Disc(K/F))^{2s} \Nm_{K/\bQ}(\Disc(L/K))^{s}} \nonumber \\
&= \sum_{[K:F]=2} \frac{1}{\Nm_{F/\bQ}(\Disc(K/F))^{2s}} \sum_{[L:K]=2} \frac{1}{\Nm_{K/\bQ}(\Disc(L/K))^{s}} \nonumber \\
&= \sum_{[K:F]=2} \frac{\Phi_{K,2}(C_2,s)}{\Nm_{F/\bQ}(\Disc(K/F))^{2s}}\label{phitildebetter}.
\end{align}

\subsection{Counting using Dirichlet series}
\label{dirichlet}

The previous section indicates that in order to study $\tilde{\Phi}(s)$, we will be compelled to deal with not just $\Phi_{F,4} (D_4, s)$, but also $\Phi_{F,4} (C_4, s)$ and $\Phi_{F,4} (V_4, s)$ that appears in \eqref{phij}, which were first considered by  Wright \cite{wright}. By \cite[Theorem I.1]{wright}, it is known that $\Phi_{F,4}(C_4,s)$ and  $\Phi_{F,4}(V_4,s)$ are absolutely convergent for $\Re(s)>1/2$. Using these functions, Wright proved that 
  $$ N_{F,4}(C_4,X) = O\left(X^{1/2}\right)$$
and
  $$ N_{F,4}(V_4,X) = O\left(X^{1/2}\log^2(X)\right).$$
The functions $\Phi_{F,4}(C_4,s)$ and $\Phi_{F,4}(V_4,s)$ can also be rewritten in terms of $L$-functions as given below.

\begin{proposition}[{\cite[Theorem 4.3]{cdoc4}}]\label{propc4}
We have that
\begin{equation}
\Phi_{F,4}(C_4,s) = \frac{2^{r_1(F)+r_2(F)-1}}{4^{2[F:\bQ]s} \zeta_F(2s)}\sum_{\mathfrak{c} | 4 \cO_F} \Nm(\mathfrak{c})^{2s} \sum_{\chi} L_F(2s,\chi) F_{\mathfrak{c},\chi}(s),
\label{c4}
\end{equation}where the sum over $\chi$ is over quadratic characters of $\Cl_{\mathfrak{c}}(F)$, the sum over $\fc$ is over ideals dividing $(4)$, and where $F_{\mathfrak{c},\chi}(s)$ is defined at the bottom of page $499$ in \cite{cdoc4} and converges for $\Re(s)>1/3.$ (Note that $F_{\mathfrak{c},\chi}(s)$  is given in terms of a sum over quadratic extensions $k/F$ which are embeddable into a $C_4$-extension.)
\end{proposition}

\begin{proposition}[{\cite[Theorem 5.1]{cdov4}}]\label{propv4}
We have that
$$\Phi_{F,4}(V_4,s) = \frac{1}{6}\Psi(s) - \frac{1}{2} \Phi_{F,2}(C_2,2s) - \frac{1}{6}$$
where
\begin{align}
\Psi(s) &= \frac{2^{2r_1(F)+2r_2(F)}}{2^{6[F:\bQ]s}} \nonumber \\
& \times \sum_{\substack{\mathfrak{n} | 2 \cO_F \\ \mathfrak{m} | 2 \cO_F}} \Nm(\mathfrak{nm})^{2s-1}   \sum_{\mathfrak{c} | \mathfrak{n}}\frac{\mu_F(\mathfrak{n}/\mathfrak{c})}{\Nm(\mathfrak{n}/\mathfrak{c})} \Nm((\mathfrak{m},\mathfrak{c}))^{2s} \prod_{\substack{\fp \text{ s.t. } \\ 1 \leq v_{\mathfrak{p}}(\mathfrak{m}) \leq v_{\mathfrak{p}}(\mathfrak{c})}}\left(1- \Nm(\mathfrak{p})^{-4s} \right) \prod_{\mathfrak{p}|\mathfrak{m}/(\mathfrak{m},\mathfrak{c})} \left(1- \Nm(\mathfrak{p})^{-2s} \right) \nonumber \\
& \times  \sum_{\chi, \psi} \sum_{\fq \mid \mathfrak{m}/(\mathfrak{m},\mathfrak{c})} L_F(2s,\chi_{1,\fq} \chi_{2,\fq}) L_F(2s,\chi_{1,\fq} \psi)L_F(2s,\chi_{2,\fq} \psi) H_\fq(s,\chi,\psi), \label{fs}
\end{align}
where $\mathfrak{n}, \mathfrak{m}$, and $\mathfrak{c}$ run over integral ideals, $\fp$ and $\fq$ run over prime ideals,  $\chi$ and $\psi$ run over all quadratic characters of $\Cl_{\fm}(F)$ and $\Cl_{\fn}(F)$ respectively, and $H_\fq(\chi,\psi,s)$ is given by an Euler product that converges for $\Re(s)>1/4$.
\end{proposition}

\begin{cor}
The poles and the residues of $\Phi_{F,4}(D_4,s)$ are the same as the poles of $\frac{1}{2}\tilde{\Phi}(s)$.
\end{cor}
\begin{proof}
This follows directly from equation \eqref{phij} and work of Wright (see \cite[Theorem I.1]{wright}) or Propositions \ref{propc4} and \ref{propv4}.
\end{proof}

It will be more convenient to work with the function
\begin{equation}
\Phi(s) = \sum_{[K:F]=2} \frac{1}{\Nm(\Disc(K/F))^{2s} }\left(1+\Phi_{K,2}(C_2,s) \right) \label{phis}
\end{equation}
rather than the function $\tilde{\Phi}(s)$ defined in \eqref{phitilde}.
From \cite[Theorem 1.1]{cdod4}, we have
\begin{equation}
1+\Phi_{K,2}(C_2,s) = \frac{2^{-r_2(K)}}{\zeta_{K}(2s)} \sum_{\mathfrak{c} \mid 2\cO_K} \Nm(2/\mathfrak{c})^{1-2s} \sum_{\chi} L_{K}(s,\chi),
\label{eq1}
\end{equation}where $\mathfrak{c}$ runs over integral ideals and $\chi$ runs over all the quadratic characters of the ray class group $\Cl_{\mathfrak{c}^2}(K)$. Note that we would not expect a summation formula in terms of $L$-functions like \eqref{eq1} to exist if we did not add 1 to $\Phi_{K,2}(C_2,s)$, and so computing the residue of $\tilde{\Phi}(s)$ would be much more difficult to do directly as we would not be able to utilize the study of $L_K(s,\chi)$. 

By equation \eqref{phij}, $\Phi(s)$ can be expressed as
\begin{equation}\label{phik}
 \Phi(s)= 2\Phi_{F,4}(D_4,s) +\Phi_{F,2}(C_2,2s)+\Phi_{F,4}(C_4,s)+3 \Phi_{F,4}(V_4,s).
\end{equation}
 
From equation \eqref{eq1} we also have 
\begin{equation}
\Phi_{F,2}(C_2,s)= -1+ \frac{2^{-r_2(F)}}{\zeta_F(2s)} \sum_{\mathfrak{c}\mid 2\cO_F} \frac{\Nm(2/\mathfrak{c})}{\Nm(2/\mathfrak{c})^{2s}}\sum_{\chi} L_{F}(s,\chi),
\label{c2}
\end{equation}
where $\mathfrak{c}$ runs over all integral ideals of $F$ dividing $2$ and $\chi$ runs over all quadratic characters of the ray class group $\mathcal{C}\ell_{\mathfrak{c}^2}(F)$ modulo $\mathfrak{c}^2$. 

We now focus on $\Phi(s)$. It follows from \eqref{eq1} and Lemma \ref{2.1.1} that $1+\Phi_{K,2}(C_2,s)$ is meromorphic in the complex plane and has a simple pole at $s=1$. If we denote by $R$ the residue of $\Phi(s)$ at $s=1$, by using the convexity bound
$$L_K(s,\chi) \ll \Nm (\Disc(K/F))^{\frac{1}{2}},$$ combined with equations \eqref{phis} and \eqref{eq1}, it follows that $\Phi(s)-R$ is absolutely convergent for $\Re(s) \geq 3/4+\epsilon.$ From equations \eqref{phik} and \eqref{c2},  in conjunction with Propositions \ref{propc4} and \ref{propv4} we also get the following.

\begin{prop}\label{mercontphif4}
The function $\Phi_{F,4}(D_4,s)$ has meromorphic continuation to the whole complex plane and has a simple pole at $s=1$. Furthermore, $\frac{1}{2}\Phi(s)$ and $\Phi_{F,4}(D_4,s)$ have the same residue at the simple pole at $s = 1$. If $\Phi(1)$ denotes the residue of $\Phi(s)$ at the pole $s=1$, then $\Phi_{F,4}(D_4,s)-\frac{\Phi(1)}{2}$ is absolutely convergent for $\Re(s) \geq 3/4+\epsilon$.
\end{prop}

\section{Proof of the main theorem}

We will now prove Theorem \ref{main}, originally stated in Remark 4 of \cite[Corollary 6.1]{cdod4}. We denote $c_{F,4}(D_4)$ the constant in the main term, i.e.
$$c_{F,4}(D_4) = \sum_{[K:F]=2} \frac{1}{2^{r_2(K)+1} \Nm(\Disc(K/F))^2} \frac{\zeta_{K}(1)}{\zeta_{K}(2)}.$$
Recall that $\zeta_{K}(1)$ denotes the value of the residue of the Dedekind zeta function of $K$ at $s=1.$  With this notation, we need to prove that
$$ N_{F,4} (D_4,X) = c_{F,4}(D_4)X + O (X^{1-\frac{1}{2d}+\epsilon})$$
where $d = [F:\mathbb{Q}]$.
By Perron's formula, we have that
\begin{equation}
N_{F,4}(D_4,X) = \frac{1}{2 \pi i} \int_{\Re(s) =2} \Phi_{F,4}(D_4,s) \frac{X^s}{s}\, ds.
\label{perron}
\end{equation}
When we shift the contour to the left to the line $\Re(s) = 1 + \epsilon$, we do not encounter any poles. So we can rewrite
$$N_{F,4}(D_4,X) = \frac{1}{2 \pi i} \int_{\Re(s) =1+\epsilon} \Phi_{F,4}(D_4,s) \frac{X^s}{s}\, ds.$$
Using standard contour integration, we get that
\begin{equation}
N_{F,4}(D_4,X) = \frac{1}{2 \pi i} \int_{1+\epsilon-iT}^{1+\epsilon+iT} \Phi_{F,4}(D_4,s) \frac{X^s}{s}\, ds+O \left(\frac{X^{1+\epsilon}}{T} \right),
\label{contour}
\end{equation}
and we will choose $T$ suitably. We shift the line of integration to $\Re(s)=3/4+\epsilon$ and encounter the simple pole at $s=1$. The residue of the pole at $s=1$ will contribute the main term which will be of size $X$, and the error term will come from bounding the contribution of the integrals on the other three sides of the contour of integration.

Next we bound the integral on the two horizontal lines of the contour and on the vertical line from $3/4+\epsilon-iT$ and $3/4+\epsilon+iT$. The integral on the horizontal lines will be bounded by $X^{1+\epsilon}/T$ and for the integral on the vertical line, we have
\begin{equation}
\label{vertical}
\frac{1}{2 \pi i} \int_{3/4+\epsilon-iT}^{3/4+\epsilon+iT} \Phi_{F,4}(D_4,s) \frac{X^s}{s}\, ds \ll X^{3/4+\epsilon},
\end{equation}
where we have used Lemma \ref{mercontphif4}.

Note that using the Lindel\"of bound \eqref{lindelof=1} to bound $\Phi(s)$ yields an upper bound of size $X^{1/4+\epsilon}$ for the vertical integrals.

Finally, we evaluate the residue of the simple pole at $s=1$, which will give us the constant $c_{F,4}(D_4)$, following \cite{cdod4}. Using \eqref{phik} and Proposition \ref{mercontphif4}, we see that, since the difference between $\Phi_{F,4}(D_4,s)$ and $\frac{1}{2} \Phi(s)$ is holomorphic for $\Re(s) > 1/2$, the two functions $\Phi_{F,4}(D_4,s)$ and $\frac{1}{2} \Phi(s)$ have the same residue at $s=1.$  From equations \eqref{l-zeta}, \eqref{phis}, \eqref{eq1}, and \eqref{phik} it follows that the value of the residue at $s=1$ is equal to
\begin{align}
\frac{1}{2} \sum_{[K:F]=2} \frac{2^{-r_2(K)} \zeta_{K}(1)}{\Nm(\Disc(K/F))^2 \zeta_{K}(2)} \sum_{\mathfrak{c}|2\cO_K} \Nm(2/\mathfrak{c})^{-1} \prod_{\mathfrak{p}|\mathfrak{c}} (1-\Nm(\mathfrak{p})^{-1}), \label{residue2}
\end{align}
where $\fc$ runs over integral ideals of $\cO_K$ and $\fp$ runs over prime ideals of $\cO_K$.
Combining \eqref{contour}, \eqref{vertical}, \eqref{residue2} and Lemma \ref{residue} and choosing $T=X^{1/4}$, the conclusion follows. 

Note that if we assume the Lindel\"of hypothesis, then using the observation following \eqref{vertical}, it follows that we can choose $T=X^{3/4}$ and then the improved bound in the second part of Theorem \ref{main} follows.
\begin{lemma}\label{residue} 
For a number field $E$, 
$$\sum_{\mathfrak{c}|2\cO_E} \Nm(2/\mathfrak{c})^{-1} \prod_{\mathfrak{p}|\mathfrak{c}} (1-\Nm(\mathfrak{p})^{-1})=1,$$
where $\fc$ runs over integral ideals of $\cO_E$ and $\fp$ runs overs prime ideals of $\cO_E$.
\end{lemma}
\begin{proof}
Following \S3.3 of \cite{cdod4}, the above formula can be proven by writing
\begin{align*}
\sum_{\mathfrak{c}|2}& \Nm(2/\mathfrak{c})^{-1} \prod_{\mathfrak{p}|\mathfrak{c}} (1-\Nm(\mathfrak{p})^{-1})= \frac{1}{\Nm(2)} \sum_{\mathfrak{c}|2}\Nm(\mathfrak{c}) \prod_{\mathfrak{p}| \mathfrak{c}} (1-\Nm(\mathfrak{p})^{-1}) \\
&= \frac{1}{\Nm(2)} \prod_{\mathfrak{p}|2} \left(1+ \sum_{j=1}^{v_{\mathfrak{p}}(2)} \Nm(\mathfrak{p}^j) (1-\Nm(\mathfrak{p})^{-1}) \right),
\end{align*}
where $v_{\mathfrak{p}}(2)$ is the valuation of $2$ at $\mathfrak{p}$. The previous formula simplifies to
\begin{align*}
\sum_{\mathfrak{c}|2}& \Nm(2/\mathfrak{c})^{-1} \prod_{\mathfrak{p}|\mathfrak{c}} (1-\Nm(\mathfrak{p})^{-1}) = \frac{1}{\Nm(2)} \prod_{\mathfrak{p}|2} \Nm(\mathfrak{p})^{v_{\mathfrak{p}}(2)}=1.
\end{align*}
\end{proof}

\section{Acknowledgements}

We would like to thank Brandon Alberts, Henri Cohen, Erik Holmes, Jordan Ellenberg, Dick Gross, Kiran S. Kedlaya, Robert Lemke Oliver, Arul Shankar, Jiuya Wang, Melanie Matchett Wood and the referee for helpful discussions. We would also like to thank Wei Ho and Renate Scheidler, the WIN5 organizers, for providing the impetus for this project.

\printbibliography
\end{document}